\documentclass[11pt]{amsart}
\usepackage[utf8]{inputenc}

\title{Links on Incompressible Surfaces and Volumes}
\author{Corbin Reid}

\usepackage{amsfonts, amstext, amsmath, amsthm, amscd, amssymb}

\usepackage{graphicx, color}

\usepackage{xcolor}

\usepackage{import}

\usepackage{microtype}

\newcommand{\bdy}{\partial}   
\newcommand{\NN}{\mathbb{N}}  

\DeclareMathOperator{\Vol}{Vol}

\newtheorem{theorem}{Theorem}[section]
\newtheorem{proposition}[theorem]{Proposition}
\newtheorem{lemma}[theorem]{Lemma}
\newtheorem{corollary}[theorem]{Corollary}

\newtheorem*{namedtheorem}{\theoremname}
\newcommand{\theoremname}{testing}
\newenvironment{named}[1]{\renewcommand{\theoremname}{#1}\begin{namedtheorem}}{\end{namedtheorem}}

\theoremstyle{definition}
\newtheorem{definition}[theorem]{Definition}

\newtheorem{remark}[theorem]{Remark}

\newtheorem{construction}[theorem]{Construction}

\definecolor{mygreen}{rgb}{0.2 0.7 0.2}

\newcommand{\cut}{\setminus\setminus}

\usepackage[hidelinks,pagebackref,pdftex]{hyperref}

\begin{document}

\begin{abstract}
    We consider volumes of two families of links that have been the focus of recent results on geometry, namely weakly generalised alternating (WGA) links and fully augmented links (FAL). Both have known lower bounds on hyperbolic volume in terms of their diagram combinatorics, but less is known about upper bounds. In fact, Kalfagianni and Purcell recently found a family of WGA knots on a compressible surface for which there can be no upper bounds on volume in terms of twist number.  They asked if upper volume bounds always exist on incompressible surfaces. We show the answer is no: we find infinite families of WGA and FALs on incompressible surfaces with no upper bound on volume in terms of twist number.
\end{abstract}

\maketitle

\section{Introduction}

    The volume of a hyperbolic 3-manifold can be a powerful invariant, but it is often difficult to determine given only a combinatorial description of a 3-manifold, such as a link diagram. One method to sidestep this difficulty is to instead bound the volume of a link complement above and below in terms of features of the diagram; the feature we consider here is the \emph{twist number}, which has been used to successfully bound the volumes of alternating knots and links from above in \cite{up_volume_bound} for links with planar projections in $S^3$, in \cite{kwon_generalize} for links in the thickened torus, and in \cite{doubledfullyaug} for general virtual links. However, in \cite{surface_volume_bound}, Kalfagianni and Purcell found a family of links in $S^3$ with no upper volume bound in terms of twist number. That family of links was projected to a genus two Heegaard surface, which is compressible; the authors asked if a volume bound may always be found in terms of twist number, provided the projection surface is incompressible \cite[Question 1.5]{surface_volume_bound}. This work answers that question in the negative. We construct a family of links that project to incompressible surfaces in the ambient manifold, the volume of which cannot be bounded above in terms of twist number. 
    
    The construction employs fully augmented links, which are constructed from general links in such a way that twist number is emphasized. Fully augmented links have been studied extensively as planar diagrams in $S^3$, for example in \cite{intro_aug_links,up_volume_bound,BlairFuterTomova,volume_bound,canon_poly_decomp}. The particular generalisation of fully augmented links in this paper, projecting to closed oriented surfaces with positive genus, enables applications to 3-manifolds beyond $S^3$, and has appeared previously in \cite{doubledfullyaug,kwon_generalize,kwon_tham_generalize, adams_generalize,construct_knots}.

    The primary result of this paper is as follows:

\begin{named}{Theorem \ref{thm:infinite_volume}}
    Let $\Sigma$ be a closed surface of genus at least two. Let $M$ be either the doubled thickened surface $\Sigma\times S^1$, or the mapping torus $(\Sigma\times I)/\phi$ where $\phi$ is a map that acts nontrivially on the isotopy class of at least one essential curve in $\Sigma$. Then there exist families of hyperbolic fully augmented links $\{J_n\}_{n\in\NN}$ such that:
    \begin{itemize}
        \item each $J_j\in\{J_n\}_{n\in\NN}$ projects to an incompressible embedding of $\Sigma$ in $M$,
        \item the number of crossing circles is a fixed natural number $c$ for all $J_j\in\{J_n\}_{n\in\NN}$, and
        \item the volume of the complement $\Vol(M\setminus J_n)$ approaches infinity as $n$ approaches infinity.
    \end{itemize}
\end{named}

In particular, this shows that there does \emph{not} exist a bound on volume in terms of the number of crossing circles, linear or otherwise.

An additional result relates this to the weakly generalised alternating links in the construction of Kalfagianni-Purcell \cite{surface_volume_bound}. It is this result that specifically answers Question 1.5 of that paper in the negative --- that is, there exist weakly generalised alternating links that project to surfaces that are incompressible in the ambient manifold, but which nevertheless do not admit a linear, or indeed any, upper bound on hyperbolic volume in terms of twist number. 


\begin{named}{Corollary \ref{cor:WGA_infinite_volume}}
    Let $\Sigma$ be a closed surface of genus at least two. There exist manifolds $M$ and families of weakly generalised alternating knots and/or links $\{K_n\}_{n\in\NN}$ such that:
    \begin{itemize}
        \item each $K_j\in\{K_n\}_{n\in\NN}$ projects to an incompressible embedding of $\Sigma$ in $M$,
        \item the twist number is a fixed natural number $c$ for all $K_j\in\{K_n\}_{n\in\NN}$, and
        \item the volume of the complement $\Vol(M\setminus K_n)$ approaches infinity as $n$ approaches infinity.
    \end{itemize}
\end{named}


\section{Fully Augmented Links in General Surfaces}

This section establishes the machinery that will be used throughout the paper. The first is the class of diagrams that will be used to investigate links in more general settings, adapted from \cite{alt_links_surfaces(chunk_decomp)}.

\begin{definition}\label{def:projection}
    Let a link $L$ lie in the neighbourhood $\Sigma\times I$ of a (possibly disconnected) closed surface $\Sigma$. The \emph{generalised projection} or \emph{generalised diagram} $\pi(L)$ of $L$ is the image of $L$ under the projection $\pi:\Sigma\times I\longrightarrow \Sigma$, with crossing information.
\end{definition}

The two common conditions for working with generalised projections below will be used throughout this paper.

\begin{definition}\label{def:cellular}
    The projection of a link is \emph{cellular} in a projection surface if its complementary regions in that surface are homeomorphic to discs. 
\end{definition}

\begin{definition}\label{def:weakly_prime}
    A link projection to a surface $\Sigma$ is \emph{weakly prime} if $D$ is any disc in $\Sigma$ where $\bdy D$ meets the projection exactly twice, either $D$ or $\Sigma\setminus D$ contains only an unknotted strand.
\end{definition}

Two main classes of links will be considered in this paper.

\begin{definition}\label{def:FAL}
    A \emph{fully augmented link} is a link that admits a projection such that each component of $L$ is one of two types:
    \begin{itemize}
        \item unknotted \emph{crossing circles} that lie perpendicular to the projection surface, each bounding a disc called a \emph{crossing disc}, and
        \item \emph{projection components} that collectively intersect each crossing disc exactly twice.
    \end{itemize}
    In this definitive projection, projection components lie entirely in the projection surface save for at most one crossing next to each crossing disc (see Figure \ref{fig:half_twist}). The number of crossing circles is denoted $c$ and the number of projection components, $l$.
\end{definition}

    \begin{figure}[h]
        \begin{center}
        \includegraphics{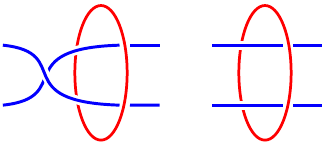}
        \end{center}
        \caption{Crossing circles (red) shown with and without a half-twist in the projection components (blue).}
        \label{fig:half_twist}
    \end{figure}

In general, diagrams of fully augmented links are assumed to be of this definitive type unless otherwise specified. All projections of fully augmented links considered in this paper are of this definitive type. Also note that if a fully augmented link projection is cellular, every projection component meets at least one crossing circle --- if a component meets no crossing circles it has no crossings, and so has a neighbourhood that is an annulus.

The definition of the second class of link below is reproduced from Howie and Purcell \cite{alt_links_surfaces(chunk_decomp)}.

\begin{definition}\label{def:WGA}
    A \emph{weakly generalised alternating} link projection on a (possibly disconnected) surface $\Sigma$:
    \begin{itemize}
        \item is weakly prime;
        \item has link components that project to each component of $\Sigma$;
        \item has at least one crossing in the projection of each link component;
        \item divides $\Sigma$ into checkerboard-colourable faces; 
        \item has respresentativity $\geq 4$, i.e.~each compression disc of $\Sigma$ has at least four transverse intersections with the link projection;
        \item alternates about each complementary region of $\Sigma$, i.e.~every strand about the edge of each complementary region of $\Sigma$ passes from an over- to an understrand.
    \end{itemize}
\end{definition}

An important feature of general links that connects them to fully augmented links:

\begin{definition}\label{def:twist_reduced}
    A \emph{twist region} of a diagram consists of either a string of bigons arranged end-to-end or a single crossing adjacent to no bigons. The minimum number of twist regions in a diagram of a given link is the \emph{twist number} of the link.
\end{definition}

We require diagrams to be alternating in twist regions; if not, the number of crossings in the diagram may be reduced in an obvious way. The connection is as follows:

\begin{definition}\label{def:augmenting}
    A link diagram is \emph{augmented} by adding an unknotted component transversely about a twist region such that it bounds a disc through the twist region. For one-crossing twist regions, there are two choices of augmentation.
\end{definition}

The complement of the link obtained by augmenting each twist region of a twist reduced link is homeomorphic to that obtained by removing all \emph{full twists} from each twist region, leaving either only a crossing circle or a crossing circle with a \emph{half-twist} (single crossing). This new link with full twists removed is the fully augmented link of Definition \ref{def:FAL}.

\begin{lemma}\label{lem:weakly-prime}
    Let $L$ be a fully augmented link lying in the neighbourhood of an incompressible surface $\Sigma$ in a 3-manifold $M$. Suppose $M\setminus L$ is hyperbolic. Then there exists a projection $\pi(L)$ of $L$ onto $\Sigma$ that is weakly prime.
\end{lemma}

    \begin{proof}
    This proof will show that given any generalised projection of a link $L$ to a surface $\Sigma$ where $L$ is hyperbolic in the ambient manifold $M$, it is possible to construct a weakly prime projection. A key tool is Thurston's hyperbolisation Theorem \cite{Thurston_Hyperbolization} which, among other consequences, states that any annulus or sphere in a hyperbolic manifold is inessential.

    Suppose there exists a disc $D\subset \Sigma$ such that $\pi(L)$ intersects $\bdy D$ exactly twice and $D$ contains something other than a trivial arc, i.e.~$\pi(L)$ is not weakly prime. First note that if $D$ contains an entire connected component of $\pi(L)$, then one can find an essential sphere in $M\setminus L$ by taking a neighbourhood $N(D)$ of $D$ and contracting it until $\bdy N(D)$ does not intersect $L$. This is a contradiction, as $M\setminus L$ is hyperbolic and so contains no essential spheres. 
    
    Take a neighbourhood $N(D)$ of $D$ in $M$ such that $\bdy N(D)$ intersects $L$ exactly twice, i.e.~$\bdy N(D)$ is an annulus. Call this annulus $A$. Given $M\setminus L$ is hyperbolic, $A$ cannot be essential. Suppose $A$ is compressible; then the intersection of the compression disc of $A$ with $D$ divides $D$ into two discs, each of which has boundary that meets $\pi(L)$ exactly once. As $L$ is composed of closed curves, this is impossible. Then $A$ is boundary compressible --- in particular, there exists an ambient isotopy in $M$ of $L\cap D$ onto an arc of $\bdy D$. As $D$ is a disc on $\Sigma$, $\bdy D$ is a simple closed curve. Thus, this isotopy removes all crossings from $L\cap D$. Repeat this for every $D\subset \Sigma$ where $\bdy D$ meets $\pi(L)$ exactly twice and $D\cap\pi(L)$ is not a trivial arc; after finitely many such isotopies, $\pi(L)$ is weakly prime.
    \end{proof}

\begin{theorem}\label{thm:augment_to_WGA}
    Let $\pi(L)$ be a cellular, weakly prime projection of a fully augmented link $L$ to an incompressible surface $\Sigma$, where $\Sigma\setminus\pi(L)$ is checkerboard-colourable. Then an integer $t_j\neq0$ may be chosen for each crossing circle of $L$ such that performing $1/t_k$ Dehn filling on the $k$-th crossing circle of $L$ for all $k$ produces a weakly generalised alternating link.
\end{theorem}

    \begin{proof}
    Recall the conditions of Definition \ref{def:WGA}. We will show each is satisfied by the link obtained by Dehn filling the crossing circles of a fully augmented link with the stated properties.

    As $\Sigma$ has no compression discs, the representativity of $\pi(L)$ is infinite. A hyperbolic fully augmented link admits a weakly prime projection by Lemma \ref{lem:weakly-prime}; it follows that the Dehn filling is weakly prime.

    A cellular diagram $\pi(L)$ necessarily projects to all components of $\Sigma$, as $\Sigma$ is composed of closed surfaces. Further, performing $1/t$ Dehn filling on a crossing circle with $t\neq0$ creates a twist region with $\geq1$ crossing, ensuring at least one crossing is projected to each component of $\Sigma$. 

    Figure \ref{fig:check-colouring} shows that Dehn filling a crossing circle of a checkerboard-colourable fully augmented link diagram produces a twist region that is checkerboard colourable in the same manner. By Definition 2.7 of \cite{alt_links_surfaces(chunk_decomp)}, choosing a checkerboard colourable $\pi(L)$ ensures that there exists a choice of sign for $t_j$ in $1/t_j$ Dehn filling that produces a link projection that alternates around each complementary region.

    \begin{figure}
        \centering
        \includegraphics{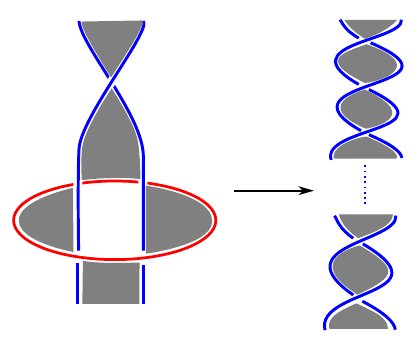}
        \caption{Converting a crossing circle to a twist region preserves checkerboard-colourability. It may be seen that the colouring is consistent with or without the half-twist.}
        \label{fig:check-colouring}
    \end{figure}

    Thus, by performing $1/t_k$ Dehn filling on the crossing circles of a fully augmented link with checkerboard colourable diagram, one can obtain a weakly generalised alternating knot or link.
    \end{proof}

\section{Constructions}\label{sec:constructions}

In this section, we adapt Kalfagianni and Purcell's construction in \cite{surface_volume_bound} to two settings with incompressible projection surfaces. 

In their construction, Kalfagianni and Purcell \cite{surface_volume_bound} drill $m$ pairs of unknotted, unlinked curves $C_1,C_2,...,C_{2m-1},C_{2m}$, about a WGA link to produce a manifold with hyperbolic volume bounded below by a specified (arbitrarily large) volume. We observe that these drilled curves are ambient isotopic to curves that bound compression discs of the projection surface in $S^3$ (Figure \ref{fig:surface_loops}), and thus the Dehn fillings of the original construction may then be realized as Dehn twists on the projection surface. 

    \begin{figure}[h]
            \begin{center}
            \includegraphics[]{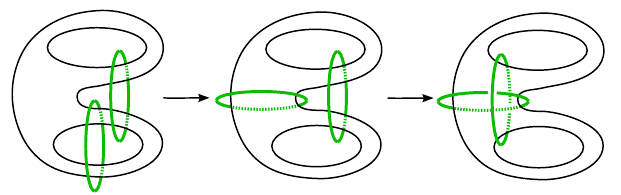}
            \end{center}
            \caption{Ambient isotopy of a pair of curves about handles of an example surface of genus $2$ to lie in a neighbourhood of the surface.}
            \label{fig:surface_loops}
        \end{figure}

\begin{construction}[Virtual link with layered curves]\label{const:layered_curves}
    Let $L$ be a fully augmented link with projection $\pi(L)$ that is cellular in a surface $\Sigma$ with genus $g\geq 2$ lying in the thickened surface $\Sigma \times I$. We will refer to $L$ as the \emph{base link}.

    Fix two essential simple closed curves $\gamma_{odd},\gamma_{even}$ on $\Sigma$, such that $\gamma_{odd}$ intersects $\gamma_{even}$ transversely, denoted $\iota(\gamma_{odd},\gamma_{even})\geq 1$. Isotope $\gamma_{odd}$ and $\gamma_{even}$ such that they do not intersect any crossing circle of $L$. Fix an integer $m$; add $m$ pairs of unknotted, unlinked curves $C_1, C_{-1}, ..., C_{m}, C_{-m}$ in a neighbourhood of $\Sigma$ as in Figure \ref{fig:layered_curves}, where:

    \begin{itemize}
        \item Each $C_i$ is isotopic to $\gamma_{odd}$ for $i$ odd and $\gamma_{even}$ for $i$ even;
        \item For $j,i$ of the same sign where $|j|>|i|$, $C_i$ lies between $C_j$ and $\Sigma$.
    \end{itemize}

    Collectively denote these \emph{layered curves} by $\mathcal{C}_m=\bigcup_{i=1}^mC_{\pm i}$, and the family of links $L\cup \mathcal{C}_m$ by $\{L_m\}_{m\in\NN}$, where $L_0$ is the base link. Each pair $C_{\pm i}$ forms the boundaries of an annulus $A_i$. Each $A_i$ intersects $\Sigma$ in a curve isotopic to either $\gamma_{odd}$ or $\gamma_{even}$ based on the parity of $i$.
\end{construction}

\begin{figure}
    \centering
    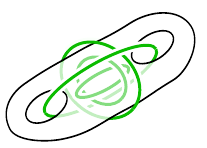
    \caption{Adding curves about a surface such that they bound annuli through the surface.}
    \label{fig:layered_curves}
\end{figure}

From this, we can now perform two separate constructions to ``close" the manifold $(\Sigma\times I)\setminus L_m$; see Figure \ref{fig:constructions}. The base link is assumed to be hyperbolic in the ambient manifold in both cases.

\begin{construction}[Doubled thickened surface]\label{const:doubled_thickened_surface}
    Begin with $L_m\subset\Sigma\times I$, as in Construction \ref{const:layered_curves}. Take a second base link $L'$ with projection $\pi(L')$ that is cellular in a copy $\Sigma'$ of $\Sigma$ lying in thickened surface $\Sigma'\times I$. We will refer to $L\cup L'$ as the base links for this manifold.

    Glue via the identity $\Sigma\times\{0\}$ to $\Sigma'\times\{0\}$ and $\Sigma\times\{1\}$ to $\Sigma'\times\{1\}$. The result is an ambient manifold isomorphic to $\Sigma\times S^1$, the \emph{doubled thickened surface}, where the base links $L\cup L'$ lie in neighbourhoods of $\Sigma$ and $\Sigma'$, and the layered curves $\mathcal{C}_m$ in lie in copies of the surface $\Sigma$ between the projection surfaces $\Sigma$ and $\Sigma'$.
\end{construction}

\begin{construction}[Mapping torus]\label{const:mapping_torus}
    Begin with the virtual link with layered curves of Construction \ref{const:layered_curves}. Identify the boundaries of $\Sigma \times I$ by a nontrivial element $\phi$ of the mapping class group of $\Sigma$. The result is the ambient $(\Sigma\times I)/\phi$, the \emph{mapping torus of $\phi$}, with the base link $L$ in a neighbourhood of $\Sigma$ and the layered curves $\mathcal{C}_m$ in copies of $\Sigma$.

    In this construction, we require that the isotopy classes of $\gamma_{odd}$ and $\gamma_{even}$ are nontrivially acted on by $\phi$, denoted $\phi(\gamma)\not\sim\pm\gamma$. This may restrict the possible choices of $\gamma_{odd}$ and $\gamma_{even}$ for some $\phi$.
\end{construction}

\begin{figure}[h]
    \centering
    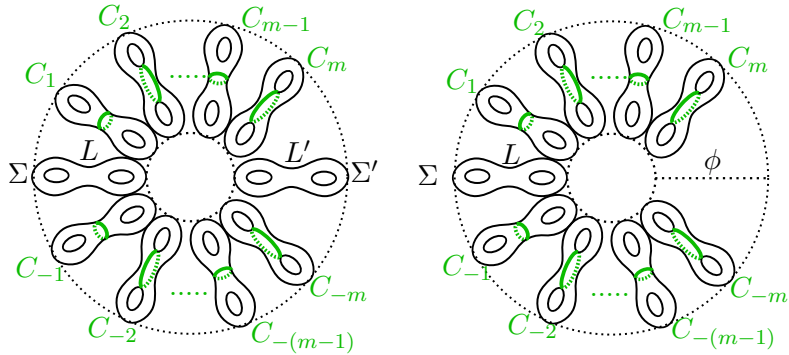
    \caption{The two constructions of $M$, represented schematically; $L$ and $L'$ in $\Sigma\times S^1$ at left, and $L$ in $(\Sigma\times I)/\phi$ at right.}
    \label{fig:constructions}
\end{figure}

\begin{lemma}\label{lem:one_curve}
    Suppose an element $\phi$ of the mapping class group of a closed surface $\Sigma$ acts nontrivially on the isotopy class of a nontrivial simple closed curve $\gamma$, i.e.~$\phi(\gamma)\not\sim\pm\gamma$. Then there exists a second such curve $\gamma'$ such that $\iota(\gamma,\gamma')\geq 1$ and $\phi(\gamma')\not\sim\pm\gamma'$. 
\end{lemma}

    \begin{proof}
        Given $\gamma$ as above, take a second curve $\alpha$ such that $\iota(\gamma,\alpha)\geq 1$. If $\phi(\alpha)\not\sim\pm\alpha$, take $\gamma'=\alpha$; done. Suppose instead that $\phi(\alpha)\sim\pm\alpha$. We will construct a curve $\gamma'$ that satisfies the proposition. 

        Take the Dehn twist of $\gamma$ about $\alpha$, $T_\alpha(\gamma)$. This has intersection number $\iota\big(\gamma,T_\alpha(\gamma)\big)=\iota(\gamma,\alpha)^2\geq1$. As $\phi$ is a homeomorphism, it preserves intersections; thus, $\iota\big(\phi(\gamma),\phi\big(T_\alpha(\gamma)\big)\big)=\iota\big(\gamma,T_\alpha(\gamma)\big)\geq 1$. We now only require that $\phi\big(T_\alpha(\gamma)\big)\not\sim T_\alpha(\gamma)$. We have that $\phi\big(T_\alpha(\gamma)\big)=T_{\phi(\alpha)}\big(\phi(\gamma)\big)=T_{\pm\alpha}\big(\phi(\gamma)\big)$, as $\phi(\alpha)\sim\pm\alpha$. 
        
        If $\phi(\alpha)\sim\alpha$, then we have that $T_\alpha(\gamma)=T_\alpha\big(\phi(\gamma)\big)$; this implies that either $\gamma\sim\alpha$, $\iota(\gamma,\alpha)=0$, or $\gamma\sim\phi(\gamma)$, all of which are false by assumption. 
        
        If $\phi(\alpha)\sim-\alpha$, then $T_\alpha(\gamma)\sim\phi\big(T_\alpha(\gamma)\big)=T_\alpha^{-1}\big(\phi(\gamma)\big)$, which implies that $T_\alpha^2\gamma\sim\phi(\gamma)$; that is, the map $\phi$ acts on $\gamma$ as two Dehn twists about $\alpha$. This, in turn, implies that $T_\alpha$ is order two, as $T_\alpha(\gamma)\sim\phi\big(T_\alpha(\gamma)\big)=T_\alpha^3(\gamma)$; this is a contradiction, as Dehn twists have infinite order. Note also that this implies $\gamma\sim T_\alpha^2(\gamma)=\phi(\gamma)$, which is again false by assumption.
    \end{proof}

\begin{remark}
    To perform Construction \ref{const:mapping_torus} in the mapping torus of a map $\phi$, it suffices to find one curve $\gamma$ such that $\phi(\gamma)\not\sim\gamma$, and construct a second as in Lemma \ref{lem:one_curve}. 
\end{remark}

Thus the only mapping tori $(\Sigma\times I)/\phi$ for which Construction \ref{const:mapping_torus} is not possible are those for which $\phi(\gamma)\sim\pm\gamma$ for all nontrivial simple closed curves $\gamma\subset\Sigma$. See Theorem \ref{thm:bounded_volume} and Remark \ref{rem:hyperelliptic_involution}.

\subsection{Hyperbolicity}\label{sec:hyperbolicity}

We require some preliminary observations to prove Constructions \ref{const:mapping_torus} and \ref{const:doubled_thickened_surface} are hyperbolic.

\begin{lemma}\label{lem:longitudinal_annulus}
    Let $\gamma$ and $\gamma'$ be two essential closed curves lying in distinct copies $\Sigma\times\{s\}=\Sigma_s$ and $\Sigma\times\{t\}=\Sigma_t$ of a surface $\Sigma$ in a thickened surface $\Sigma\times I$. Suppose there exists an essential annulus $A$ in $\Sigma\times I\setminus \big(N(\gamma)\cup N(\gamma')\big)$ with boundary on $\bdy N(\gamma)\cup \bdy N(\gamma')$; then $\bdy A$ is composed of curves parallel to $\bdy N(\gamma)\cap\Sigma_s$ and $\bdy N(\gamma')\cap\Sigma_t$.
\end{lemma}

    \begin{proof}
    Let $A\cap N(\gamma)\subset\bdy A$ form a nontrivial curve on $N(\gamma)$. We will consider arcs of intersection of $A$ with the surface $\Sigma_s$. Let $\alpha$ be such an arc in $A\cap\Sigma_s$. As $N(\gamma')\cap\Sigma_s$ is empty, both ends of $\alpha$ are on $N(\gamma)$; thus, $\alpha$ co-bounds a disc on $A$ with an arc of $\bdy A$. Take an edgemost disc $D$ on $A$ such that the arc of $\bdy A$ is a single connected arc $\beta$. As $\beta$ is contained in $N(\gamma)\setminus\Sigma_s$, it can be isotoped to lie on $\gamma\subset\Sigma_s$. Then $\bdy D$ bounds a disc $E$ on $\Sigma_s$, as $\Sigma_s$ is incompressible. Gluing $D$ and $E$ through $\bdy D$ forms a sphere in $\Sigma\times I$. As $\Sigma\times I$ is irreducible, this sphere bounds a ball $B$. As $D$ is assumed to be edgemost, it can be removed by isotopy of $A\cap B$ through $B$. After finitely many such isotopies, all arcs of intersection in $A\cap\Sigma_s$ are removed.

    As all arcs of intersection in $A\cap\Sigma_s$ have been removed, i.e. $A\cap\Sigma_s$ is disjoint from $\bdy N(\gamma)$, components of $A\cap N(\gamma)\subset\bdy A$ lie completely on one side of $\Sigma_s$, parallel to $\bdy N(\gamma)\cap\Sigma_s$. If $\bdy A\subset\bdy N(\gamma)$, both components of $A$ are such curves on $\bdy N(\gamma)$, as $A$ has no self intersections; if $\bdy A\cap N(\gamma')$ is nonempty, it is a curve parallel to $\bdy N(\gamma')\cap\Sigma_t$ by identical argument on $A\cap\Sigma_t$.
    \end{proof}

\begin{lemma}\label{lem:only_annulus_obvious}
    Let $\gamma,\gamma'$ lie in $\Sigma\times I$ such that there exists an essential annulus $A$ in $\Sigma\times I\setminus\big(N(\gamma)\cup N(\gamma')\big)$ with boundary on $\bdy N(\gamma)\cup \bdy N(\gamma')$, as in Lemma \ref{lem:longitudinal_annulus}. Then $\gamma$ and $\gamma'$ are in the same isotopy class on $\Sigma$, and $\bdy A$ either:
    \begin{itemize}
        \item intersects both of $N(\gamma),N(\gamma')$, and is isotopic to $\gamma\times(s,t)$, or
        \item is contained in $\bdy N(\gamma)$, and $A$ co-bounds a solid torus containing $\gamma\times(s,t)$ that has core curve $\gamma'$ with an annulus $A'\subset\bdy N(\gamma)$, up to relabelling of $\gamma,\gamma'$.
    \end{itemize}
\end{lemma}

    \begin{proof}
        Consider the first case. The annulus $A$ describes an isotopy between $\gamma$ and $\gamma'$ as by Lemma \ref{lem:longitudinal_annulus}, $\bdy A$ is composed of a longitude of each of $N(\gamma)$ and $N(\gamma')$, parallel to $\bdy N(\gamma)\cap\Sigma_s$ and $\bdy N(\gamma')\cap\Sigma_t$, respectively. As the fundamental group $\pi_1(\Sigma\times I)$ of $\Sigma\times I$ is the same as $\pi_1(\Sigma)=\pi_1(\Sigma_s)=\pi_1(\Sigma_t)$, $\gamma$ and $\gamma'$ are in the same isotopy class on $\Sigma$.

        Consider $A\cap\big(\gamma\times(s,t)\big)$. As $\bdy A$ is composed of longitudes of $N(\gamma),N(\gamma')$ as above, there are no arcs in $A\cap \big(\gamma\times(s,t)\big)$. If a component of $A\cap \big(\gamma\times(s,t)\big)$ bounds a disc $D$ in $\gamma\times(s,t)$, it also bounds a disc $E$ in $A$, as $A$ is incompressible, and vice-versa; take an innermost such disc on $A$. Gluing $D$ to $E$ through $A\cap \big(\gamma\times(s,t)\big)$ forms a sphere in $\Sigma\times I$, which bounds a ball as $\Sigma\times I$ is irreducible. Isotoping $E$ through this ball removes the trivial component of $A\cap\gamma\times(s,t)$, and after finitely many such isotopies, $A\cap\gamma\times(s,t)$ has no trivial components. Thus $A\cap \big(\gamma\times(s,t)\big)$ is composed of curves that are nontrivial on both $A$ and $\gamma\times(s,t)$. Thus $A\cap\big(\gamma\times(s,t)\big)$ divides $A$ and $\gamma\times(s,t)$ into parallel annuli. Take an innermost sub-annulus $A'$ of $A$ where $\gamma\subset\bdy A'$. Then $A'$ co-bounds a solid torus $T$ disjoint from $\gamma$ in $\Sigma\times I$ with a sub-annulus $A''$ of $\gamma\times(s,t)$. The sub-annulus $A'$ may be isotoped past $A''$ through $T$ to remove a component of $A\cap\big(\gamma\times(s,t)\big)$, except if $\gamma'\subset T$. In that case, take a similar innermost annulus $\bar A'$ in $A$ such that $\gamma'\subset \bdy \bar A'$, co-bounding a solid torus $T'$ with a sub-annulus $\bar A''$ of $\gamma\times(s,t)$; then $\gamma\not\subset T'$ and a similar isotopy of $\bar A'$ is possible. To see this, note $N(\gamma)\not\subset T$, so any solid torus containing $\gamma$ has boundary that is not contained in $T$, but $\bar A'$ originates inside $T$, and so must exit $T$ via $A''\subset\gamma\times(s,t)$; then any $\bar A'\not\subset T$ is not innermost, so $\gamma\not\subset T'$ and isotoping $\bar A'$ past $\bar A''$ through $T'$ removes components of $A\cap\big(\gamma\times(s,t)\big)$. Thus, after finitely many isotopies, $A\cap\big(\gamma\times(s,t)\big)$ is empty, and thus $A$ co-bounds a solid torus in $(\Sigma\times I)\setminus\big(N(\gamma)\cup N(\gamma')\big)$ with $\gamma\times(s,t)$, through which $A$ may be isotoped into the form of $\gamma\times(s,t)$.

        Now consider the second case, $\bdy A\subset \bdy N(\gamma)$. Again by Lemma \ref{lem:longitudinal_annulus}, $\bdy A$ is composed of two longitudes of $N(\gamma)$, and thus $A$ co-bounds a solid torus $T$ in $\Sigma\times I$ disjoint from $\gamma$ with an annulus in $\bdy N(\gamma)$; note $\gamma'\subset T$, as $A$ is essential. Trivial curves of $A\cap\Sigma_t$ are removable as both surfaces are incompressible; isotope $A$ such that $A\cap\Sigma_t$ is composed of nontrivial curves in $A$ parallel to $\bdy A$. As $\bdy A$ is composed of longitudes of $N(\gamma)$, any sub-annulus of $A$ describes an isotopy of $\gamma$, and as such any nontrivial component of $A\cap\Sigma_t$ is isotopic to $\gamma$ in $\Sigma$. As $\gamma'$ is a nontrivial curve and $\gamma'\subset T\cap\Sigma_t$, $\gamma'$ is contained in an annulus with boundaries isotopic to $\gamma$, and thus $\gamma'$ is in the same isotopy class as $\gamma$ on $\Sigma$.

        As $\gamma$ and $\gamma'$ are in the same isotopy class, there exists an annulus $A'\subset T$ which describes the isotopy between them, of which $T$ forms a neighbourhood. By the above, $A'$ may be isotoped into the form $\gamma\times(s,t)$; performing a comparable isotopy on $A$ gives the required result.
    \end{proof}

\begin{lemma}\label{lem:only_torus_obvious}
    Let $M=\Sigma\times I\setminus\big(N(\gamma)\cup N(\gamma')\big)$ is as in Lemma \ref{lem:longitudinal_annulus}. Then $M$ is toroidal if and only if $\gamma$ and $\gamma'$ are in the same isotopy class on $\Sigma$, in which case the only essential torus contains both $\gamma$ and $\gamma'$, and is isotopic to the boundary of a neighbourhood of $\gamma\times(s,t)$.
\end{lemma}

    \begin{proof}
        As $\Sigma\times I$ is atoroidal, any essential torus in $(\Sigma\times I)\setminus\big(N(\gamma)\cup N(\gamma')\big)$ bounds a solid torus $T\subset\Sigma\times I$ that contains $\gamma$ and/or $\gamma'$. Let $\gamma\subset T$; then $\gamma$ is homotopic to the core curve of $T$. As $\gamma$ is a nontrivial curve on an incompressible surface, the component of $T\cap\Sigma_s$ containing $\gamma$ is an annulus up to isotopy, and thus nontrivial components of $\bdy T\cap\Sigma_s$ are isotopic to $\gamma$ in $\Sigma_s$. If $\gamma'\not\subset T$, $\bdy T$ is boundary parallel, which is a contradiction; thus $\gamma'\subset T$. Then, by identical argument, nontrivial components of $\bdy T\cap\Sigma_t$ are isotopic to $\gamma'$. As $\Sigma_s$ is disjoint from $\Sigma_t$, nontrivial curves in $\bdy T\cap\Sigma_s$ are parallel in $\bdy T$, and thus isotopic, to those  to $\bdy T\cap\Sigma_t$. This gives an isotopy between $\gamma$ and $\gamma'$ in $\Sigma\times I$, and thus $\gamma$ and $\gamma'$ are in the same isotopy class in $\Sigma$.

        By Lemma \ref{lem:only_annulus_obvious}, there exists an essential annulus $\gamma\times(s,t)$ in $M$ when $\gamma$ and $\gamma'$ are in the same isotopy class on $\Sigma$. As $\gamma,\gamma'$ are isotopic to the core curve of $T$, $T$ is isotopic to a neighbourhood of $\gamma\times(s,t)$, as required. 
    \end{proof}
    
\begin{lemma}\label{lem:FAL_restrict_ess_surface}
    Let $L$ be a fully augmented link in a manifold $M$. Suppose $S$ is an essential surface with boundary, with components of $\bdy S$ on a link component $L^*\subset L$ in $M\setminus L$. Then either:

    \begin{itemize}
        \item components of $\bdy S\cap L^*$ are pure meridians or longitudes on $N(L^*)$ when $L^*$ is a projection component or crossing circle, respectively, or
        \item $S$ nontrivially intersects at least one crossing disc $D$.
    \end{itemize}
\end{lemma}

    \begin{proof}
        Given a crossing disc $D$, $\bdy N(L)\cap D$ is a longitude for the crossing circle bounding $D$ and a meridian for the projection components intersecting $D$. Now let $D$ be a crossing disc that intersects $L^*$; if $S$ is disjoint from $D$, $N(L^*)\cap\bdy S$ is isotopic to $\bdy N(L^*)\cap D$, i.e.~a longitude if $L^*$ is a crossing circle and a meridian if $L^*$ is a projection, as in the statement of the lemma. If not, the curve $N(L^*)\cap\bdy S$ has nonzero intersection number with $N(L^*)\cap D$, and $S\cap D$ is nonempty. 
    \end{proof}

We return to $\{L_m\}_{m\in\NN}$ as obtained by Constructions \ref{const:doubled_thickened_surface} and \ref{const:mapping_torus}. 

\begin{proposition}\label{prop:hyperbolic_LuC}
    Assume the base links $L$ or $L\cup L'$ have hyperbolic complement in their respective manifolds obtained in Constructions \ref{const:doubled_thickened_surface} and \ref{const:mapping_torus}. The families of links $\{L_m\}_{m\in\NN}=L\cup \mathcal{C}_m$ obtained in these constructions, i.e.~ the base links with layered curves added, also have hyperbolic complements.
\end{proposition}

    \begin{proof}
    Let $M$ be either $\Sigma\times S^1$ or $(\Sigma\times I)/\phi$ containing links as specified in Constructions \ref{const:doubled_thickened_surface} and \ref{const:mapping_torus}. The base curves ($m=0$) of both constructions are hyperbolic by assumption. Let $m>0$; that is, drill the layered curves $\mathcal{C}_m$ of Construction \ref{const:layered_curves}. This proof with proceed using Thurston's hyperbolization theorem \cite{Thurston_Hyperbolization}, ruling out the existence of essential discs, spheres, annuli and tori. As $(\Sigma\times S^1)\setminus (L\cup L')$ and $(\Sigma\times I)/\phi\setminus L_0$ are hyperbolic, they contain none of these essential surfaces. Refer to the projection surfaces of $L_0$ and $L_0'$ as $\Sigma_0$ and $\Sigma_0'$, respectively.

    Suppose $M\setminus L_m$ is reducible. As $M\setminus L$ is irreducible, a reducing sphere in $M\setminus L_m$ bounds a ball in $M\setminus L$ that contains components of $\mathcal{C}_m$ in $M\setminus L_m$. This is a contradiction, as all $C_{i}\subset\mathcal{C}_m$ are essential curves in incompressible surfaces and thus not contained in any simply connected region. Thus, $M\setminus L_m$ is irreducible.

    Suppose $M\setminus L_m$ is boundary reducible. All components of $\mathcal{C}_m$ are essential curves on incompressible surfaces in $M$, and so do not bound discs; this rules out any disc with boundary a curve with longitudinal components on a neighbourhood of $C_i$. Any essential disc with boundary a meridian or trivial curve on a neighbourhood of $C_i$ may be capped by a disc in $M$ by trivially Dehn filling $C_i$ to find a reducing sphere for $M\setminus L_0$, which is a contradiction. Any reducing disk then has boundary on $L_0$. Such a disc must co-bound a ball in $M\setminus L_0$ with a disc of $\bdy N(L_0)$, as $M\setminus L_0$ is boundary irreducible; but as all $C_i$ are essential curves in incompressible surfaces of $M$, no component of $\mathcal{C}$ is contained in a ball, so such a disc cannot be essential in $M\setminus L_m$.

    As $M\setminus L_0$ is anannular, an essential annulus in $M\setminus L_m$ either has least one boundary on $\mathcal{C}_m$ or is boundary parallel in $M\setminus L_0$. Consider the former.

    By construction, all components of $\mathcal{C}_m$ sit in distinct copies of the surface $\Sigma$. By Lemma \ref{lem:only_annulus_obvious}, there exists an annulus between two of these components only if they are isotopic to the same curve on $\Sigma$, i.e.~of the same parity. Any annulus through $\Sigma_0$ is punctured by $L$ as $\pi(L)$ is cellular, and any annulus between two same-parity components of $\mathcal{C}_m$ is punctured by at least one opposite-parity $C_i$ between them. 

    If $M=\Sigma\times S^1$, there exist annuli between same-parity components of $\mathcal{C}_m$ through $\Sigma_0'$. These annuli are punctured by $L'$, as $\pi(L')$ is cellular. If $M=(\Sigma\times I)/\phi$, $\phi(\gamma_{odd})\not\sim\gamma_{odd}$ and $\phi(\gamma_{even})\not\sim\gamma_{even}$ by construction; however, there exists an annulus disjoint from $\Sigma_0$ between opposite-parity and -sign components of $\mathcal{C}_m$ if $\phi(\gamma_{even})\sim\gamma_{odd}$, or vice-versa. This annulus is punctured by the paired curve of the higher-magnitude index; for example, suppose there exists an annulus between $C_{-m}$ and $C_{m-1}$ with $m$ even, i.e.~$\phi(\gamma_{even})\sim\gamma_{odd}$. Then the annulus is punctured by $C_m$, as $\gamma_{even}\not\sim\gamma_{odd}$.

    Suppose an essential annulus $A$ has both boundaries on one $C_i$. If $A\cap\Sigma_0$ is nonempty, all components bound discs on $\Sigma_0$, as $\pi(L_0)$ is cellular. As $A$ is incompressible, $A\cap\Sigma_0$ must also bound a disc on $A$; as $M\setminus L_0$ is irreducible, these discs co-bound a sphere and so the innermost such discs are removable by isotopy of $A$. After finitely many isotopies, $A$ is disjoint from $N(\Sigma_0)$; similarly, $A$ is disjoint from $N(\Sigma_0')$. Then $A$ is as described in Lemma \ref{lem:only_annulus_obvious}, and bounds a solid torus that contains a $C_j$ where $i,j$ are of the same parity; it is then punctured by at least one $C_k$ of the opposite parity.
    
    Now suppose there exists an essential annulus $A$ between a $C_i$ and a component $L_\star$ of $L$ or $L'$. If $\bdy A$ is a meridian of $N(C_i)$, Dehn filling $C_i$ gives a compression disc for $L_0$, which is a contradiction. Thus $\bdy A\cap N(C_i)$ has a longitudinal component on $N(C_i)$. As $\pi(L)$ is cellular, $L^*$ either meets or is a crossing circle. By Lemma \ref{lem:FAL_restrict_ess_surface}, if $A$ does not intersect a crossing disc, $\bdy A\cap N(L^*)$ is either a longitude of a crossing circle or a meridian of a projection component. Trivially Dehn filling $L^*$ in $M$ caps $A$ with a disc in either case, producing a compression disc for $C_i$, which is a contradiction. Thus, again by Lemma \ref{lem:FAL_restrict_ess_surface}, $\bdy A\cap N(L^*)$ nontrivially intersects a crossing disc; but, by identical argument to Lemma \ref{lem:only_annulus_obvious}, $\bdy A\cap N(C_i)$ is a longitude, and describes an isotopy of $C_i$ onto $\bdy N(L^*)$ --- as $C_i$ lies outside a neighbourhood of $\Sigma$, $C_i$ cannot intersect a crossing disc, and this is a contradiction.

    Now suppose an essential annulus $A$ in $M\setminus L_m$ has both boundaries on $L_0$. If $\bdy A$ is on two components of $L_0$, this is an essential annulus in $M\setminus L_0$, which is a contradiction; thus $\bdy A$ is on a component $L^*\subset L_0$. As $M\setminus L_0$ is hyperbolic, $A$ bounds a solid torus containing components of $\mathcal{C}_m$. By similar argument to the above, $A$ cannot intersect a crossing disc as $C_i$ is disjoint from $N(\Sigma)$; then both components of $\bdy A$ are trivial in $M$, and trivially Dehn filling $L^*$ caps $A$ with two discs to find a sphere containing $C_i$; as previous, this is a contradiction.
    
    Finally, consider essential tori. Any essential torus $T$ in $M\setminus L_m$ either bounds a solid torus or is boundary parallel in $M\setminus L_0$, as $M\setminus L_0$ is atoroidal; thus $T$ encloses component(s) of $\mathcal{C}_m$. Suppose $T$ bounds a solid torus. If $T\cap\Sigma_0$ is nonempty, all components bound discs on $\Sigma_0$, as $\pi(L_0)$ is cellular. As $T$ is incompressible, $T\cap\Sigma_0$ must also bound a disc on $T$; as $M\setminus L_0$ is irreducible, these discs co-bound a sphere and so the innermost such discs are removable by isotopy of $T$. After finitely many isotopies, $T$ is disjoint from $N(\Sigma_0)$; similarly, $T$ is disjoint from $N(\Sigma_0')$. Then $T$ is as described in Lemma \ref{lem:only_torus_obvious}, and bounds a solid torus that contains a $C_j$ where $i,j$ are of the same parity; it is then punctured by at least one $C_k$ of the opposite parity. Suppose then $T$ is boundary parallel in $M\setminus L_0$. Then, by similar argument to lemma \ref{lem:only_torus_obvious}, $T$ bounds a solid torus in $M$ with core curve isotopic to a curve $C_i$. As $T$ is boundary parallel to $L_0$, $T\cap \Sigma_0$ then contains nontrivial curves; however, this contradicts $L_0$ being cellular. 

    As $M\setminus L_m$ contains no essential discs, spheres, annuli, or tori for all $m$, it is hyperbolic.
    \end{proof}

\begin{remark}
    The condition that $\Sigma$ is of genus at least two is required for Proposition \ref{prop:hyperbolic_LuC}. Performing an analogous construction when $\Sigma$ is genus zero or one produces a manifold with essential spheres or tori, respectively, parallel to $\Sigma$, and by Thurston hyperbolization \cite{Thurston_Hyperbolization}, links on $\Sigma$ in these manifolds are not hyperbolic.
\end{remark}

To end this section, note that requiring a hyperbolic base link in Constructions \ref{const:mapping_torus} and \ref{const:doubled_thickened_surface} is not a strenuous or unreasonable assumption.

\begin{lemma}
    There exists a cellular fully augmented link $L_0$ such that $(\Sigma\times I)/\phi\setminus L_0$ is hyperbolic, where $\phi$ is as in Construction \ref{const:mapping_torus}. Similarly, there exist cellular fully augmented links $L_0$ and $L_0'$ such that $\Sigma\times S^1\setminus (L_0\cup L_0')$ as in Construction \ref{const:doubled_thickened_surface} is hyperbolic.
\end{lemma}

    \begin{proof}
        A set of sufficient conditions for hyperbolicity of fully augmented links with cellular projections to incompressible surfaces is given in Reid \cite{general_angled_chunks}, along with hyperbolicity for much more general fully augmented links. The base links of Construction \ref{const:layered_curves} --- that is, fully augmented links in thickened surfaces, also known as virtual fully augmented links --- can also be proven to be hyperbolic by the techniques of \cite{adams_generalize} and \cite{doubledfullyaug}. Further, by \cite{doubledfullyaug}, hyperbolicity of the base links of Construction \ref{const:doubled_thickened_surface} follows on observing $(\Sigma\times S^1)\setminus(L_0\cup L_0')$ is a doubled manifold; for example, this is immediate in the case $L_0$ and $L_0'$ are fully augmented links without half twists.

        To show there exist hyperbolic examples of the base links of Construction \ref{const:mapping_torus}, start with hyperbolic $(\Sigma\times I)\setminus L_0$ and denote the image of $\bdy(\Sigma\times I)$ in $(\Sigma\times I)/\phi$ by $\Sigma_\phi$. The argument follows similarly to Proposition \ref{prop:hyperbolic_LuC}, using Thurston hyperbolization.
        
        Suppose there exists an essential disc $D\subset(\Sigma\times I)/\phi\setminus L_0$. Then $\Sigma_\phi\cap D$ is composed of trivial curves on $D$; as $\Sigma_\phi$ is incompressible in $(\Sigma\times I)/\phi$, each component of $\Sigma_\phi\cap D$ also bounds a disc in $\Sigma_\phi$. As each $M$ is irreducible, $D$ co-bounds a ball and the innermost such $D$ is removable by isotopy of $D$. After finitely many such isotopies, $\Sigma_\phi\cap D$ is empty, i.e.~$D$ is an essential disc in $\Sigma\times I\setminus L_0$, contradicting the hyperbolicity of $(\Sigma\times I)\setminus L_0$. By identical argument, if there exists an essential sphere $S\subset(\Sigma\times I)/\phi$, all components of $\Sigma_\phi\cap S$ are removable by isotopy of $S$, implying $S$ is essential in $(\Sigma\times I)\setminus L_0$, which is a contradiction.

        Suppose there exists an essential annulus $A\subset(\Sigma\times I)/\phi\setminus L_0$. If $\Sigma_\phi\cap A$ is entirely removable by isotopy of $A$, then again the hyperbolicity of $(\Sigma\times I)/\phi\setminus L_0$ is contradicted; thus $\Sigma_\phi\cap A$ is composed of non-removable curves that are nontrivial in both $A$ and $\Sigma_\phi$. Cutting the manifold along $\Sigma_\phi$ then cuts $A$ into essential annuli in $(\Sigma\times I)/\phi\setminus L_0$ with boundary components on $\bdy(\Sigma\times I)$, which is again a contradiction. By identical argument, an essential torus $T\subset(\Sigma\times I)/\phi\setminus L_0$ is either isotopic to an essential torus disjoint from $\Sigma_\phi$ or is cut along $\Sigma_\phi$ into essential annuli in $(\Sigma\times I)\setminus L_0$, each of which is again a contradiction.

        As $(\Sigma\times I)/\phi\setminus L_0$ contains no essential discs, spheres, annuli, or tori, it is hyperbolic; thus there exist hyperbolic base links obtained from Construction \ref{const:mapping_torus}. Following an identical argument for Construction \ref{const:doubled_thickened_surface} starting with hyperbolic $(\Sigma\times I)\setminus L_0$ and $(\Sigma\times I)\setminus L_0'$ provides an alternate proof of the existence of hyperbolic base links for Construction \ref{const:doubled_thickened_surface}.
    \end{proof}

\subsection{Main Results}\label{sec:main_result}

To obtain the desired result, we employ annular Dehn filling. 

\begin{definition}\label{def:annular_Dehn_filling}
    Given two link components that co-bound an annulus $A$ in a manifold $M$, an \emph{annular Dehn filling} or $1/t$-\emph{annular Dehn filling} consists of performing a $1/t$ Dehn filling on one link component and a $-1/t$ Dehn filling on the other for some $t\in \NN$. The effect is of a Dehn twist through a neighbourhood of $A$; i.e. for a transverse curve $\gamma$ through $A$, $1/t$-annular filling of $A$ acts as the identity outside a neighbourhood of $A$ and spins the curve $\gamma$ a total of $t$ times about the core of $A$.
\end{definition}

In particular, annular filling that intersects a surface in a nontrivial curve acts as a Dehn twist on that curve in the surface.

Given the family of links $\{L_m\}_{m\in\NN}$ in $\Sigma\times S^1$ or $(\Sigma\times I)/\phi$ as in Constructions \ref{const:doubled_thickened_surface} and \ref{const:mapping_torus}, for each $i\in\{1,...m\}$, select an integer $t_i>0$. Perform $1/t_i$ annular Dehn filling along each $A_i$; by Definition \ref{def:annular_Dehn_filling}, this has the effect of $t_i$ full twists on $L$ along $\gamma_{odd}$ or $\gamma_{even}$ as appropriate. All choices of $m$ and $t_i$ give a countable collection of links $\{J_n\}_{n\in\NN}$ in $M$. In general, there is such a family for any appropriate choice of $L$ and either $L'$ or $\phi$ (depending on the construction) with compatible choices of $\gamma_{odd}$ and $\gamma_{even}$ on $\Sigma$.

\begin{theorem}\label{thm:infinite_volume}
Let $\Sigma$ be a closed surface of genus at least two. Let $M$ be either the doubled thickened surface $\Sigma\times S^1$ or the mapping torus $(\Sigma\times I)/\phi$ of a map $\phi$ that acts nontrivially on the isotopy class of at least one essential curve in $\Sigma$. There exist families of hyperbolic fully augmented links $\{J_n\}_{n\in\NN}$ such that:
    \begin{itemize}
        \item each $J_j\in\{J_n\}_{n\in\NN}$ projects to an incompressible embedding of $\Sigma$ in $M$,
        \item the number of crossing circles is a fixed natural number $c$ for all $J_j\in\{J_n\}_{n\in\NN}$, and
        \item the volume of the complement $\Vol(M\setminus J_n)$ approaches infinity as $n$ approaches infinity.
    \end{itemize}
\end{theorem}

    \begin{proof}
    Fix $V>0$. Fix an integer $m>0$ such that $2mv_{tet}>V$, where $v_{tet}$ is the volume of a regular ideal hyperbolic tetrahedron. The manifold $M\setminus L_m$, as defined above, is hyperbolic by Proposition \ref{prop:hyperbolic_LuC}, and has either $l+c+2m$ or $l'+c'+l+c+2m$ cusps when $M$ is the mapping torus or doubled thickened surface, respectively. By Adams \cite{n-cusp_volume}, the volume satisfies:
    \[\Vol(M\setminus L_m)>(l+c+2m)v_{tet}>2mv_{tet}>V,\]
    when $M=(\Sigma\times I)/\phi$, and a similar relation holds when $M=\Sigma\times S^1$. 

    Consider $J_j\in \{J_n\}_{n\in\NN}$ obtained by annular Dehn filling the $2m$ cusps $C_1,C_{-1},...C_m,C_{-m}\subset L_m$. For $1,...,m$, let $t_i$ be such that $J_j$ is obtained by $1/t_i$ annular Dehn filling $A_i, i\in\{1,...,m\}$. As the Dehn filling coefficients $t_i$ approach infinity, the volume of $M\setminus J_j$ approaches $\Vol(M\setminus L_m)$ from below, since Dehn filling is known to decrease volume from work of Jorgensen and Thurston \cite{Thurston_Hyperbolization}; see also work of Futer, Kalfagianni, and Purcell \cite{volume_bound}. Thus, for infinitely many $J_j\in\{J_n\}_{n\in\NN}$, the volume of the complement in $M$ is strictly greater than $2mv_{tet}>V$ for arbitrary positive $V$.

    By construction, $\gamma_{odd}$ and $\gamma_{even}$ have no intersection with crossing circles, and Dehn twists on $\Sigma$ introduce no new crossings on $\Sigma$. Thus the number of crossing circles is a constant $c$ for all $J_j$.
    \end{proof}

We may relate this result to weakly generalised alternating links as below.

\begin{corollary}\label{cor:WGA_infinite_volume}
Let $\Sigma$ be a closed surface of genus at least two. There exist manifolds $M$ and families of weakly generalised alternating knots and/or links $\{K_n\}_{n\in\NN}$ such that:
    \begin{itemize}
        \item each $K_j\in\{K_n\}_{n\in\NN}$ projects to an incompressible embedding of $\Sigma$ in $M$,
        \item the twist number is a fixed natural number $c$ for all $K_j\in\{K_n\}_{n\in\NN}$, and
        \item the volume of the complement $\Vol(M\setminus K_n)$ approaches infinity as $n$ approaches infinity.
    \end{itemize}
\end{corollary}

    \begin{proof}
    By Theorem \ref{thm:augment_to_WGA}, each hyperbolic fully augmented link $L$ with cellular projection $\pi(L)$ whose complementary regions are checkerboard-colourable on a surface $\Sigma$ is associated to a weakly generalised alternating link $K$ via $1/t_k$ Dehn filling of crossing circles. Perform these Dehn fillings on the crossing circles of appropriately chosen $L\subset M$. Suppose the resulting link is hyperbolic; again by \cite{volume_bound}, the volume of $M \setminus K$ is bounded above by the volume of $M\setminus L$ as each $t_k\longrightarrow\infty$. Add $C_1,C_{-1},...,C_m,C_{-m}$ as above, and again $1/t_i$ annular Dehn fill each $A_i$ to obtain a countable family of weakly generalised alternating links $\{K_n\}_{n\in\NN}$. By identical argument to Theorem \ref{thm:infinite_volume}, $\Vol(M \setminus K_n)\longrightarrow\infty$ as $n$ approaches infinity, and twist number remains constant, as no new crossings are introduced by annular Dehn filling.
    \end{proof}

\subsection{Bounded Case and Open Questions}

The complement of a hyperbolic fully augmented link in the trivial mapping torus (i.e.~Construction \ref{const:mapping_torus} where $\phi$ is the identity and $m=0$) \emph{does} admit a linear upper bound on hyperbolic volume. To obtain this bound, we use a natural decomposition of fully augmented link projections.

\begin{proposition}\label{prop:decomp}
    Let $L$ be a fully augmented link with a projection $\pi(L)$ onto a (possibly disconnected) closed, orientable surface $\Sigma$ embedded in a manifold $M$. There exists a decomposition of the link complement $M\setminus L$ into a collection of manifolds with boundary where each such manifold is isotopic to a component of $M\cut\Sigma=M\setminus N(\Sigma)$, where the boundary components are decorated with an ideal checkerboard-colourable graph that depends on $\pi(L)$.
\end{proposition}

    \begin{proof}
    This decomposition is analogous to a standard decomposition for fully augmented links which appears in, for example, \cite{intro_aug_links}, but this decomposition allows generalised projection surfaces.
    
    The steps to decompose a fully augmented link projection $\pi(L)$ are as follows:
        \begin{enumerate}
            \item Remove all half-twists from $\pi(L)$.
            \item Cut along $\Sigma$, bisecting each crossing circle and the twice-punctured disc it bounds.
            \item Cut along each half of each twice-punctured disc.
            \item Flatten each of the twice-punctured disc halves to the boundary.
            \item Collapse each link component to an ideal vertex.
        \end{enumerate}
    These steps are shown for a single crossing circle in Figure \ref{fig:bowtie_decomposition}.

        \begin{figure}[h]
            \begin{center}
            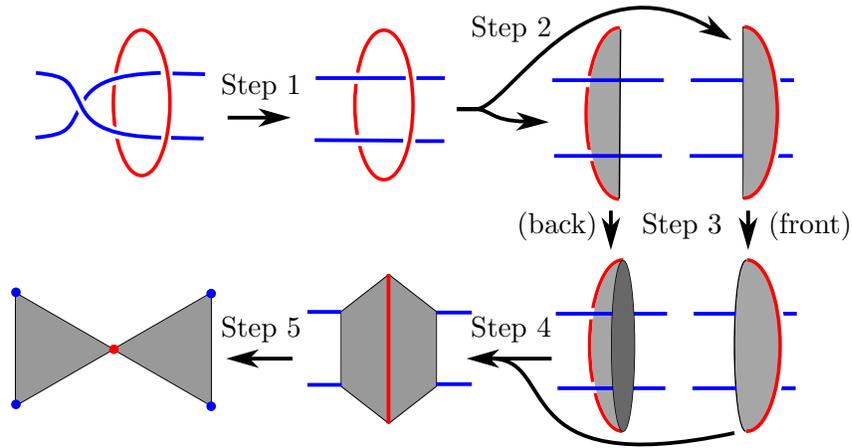
            \end{center}
            \caption{The steps of the bowtie decomposition of a fully augmented link with projection $\pi(L)$ on a surface $\Sigma$, displayed for a single crossing circle. Though the resulting image is identical for both components, note it describes two separate objects with interiors on opposite sides of the page.}
            \label{fig:bowtie_decomposition}
        \end{figure}
    
    In this process, the components of the projection surface become the white faces and the twice punctured discs become the shaded faces. All vertices of this decomposition are ideal. Each boundary arising from $\Sigma$ is decorated with such a graph.

    The manifold may be recovered from the above decomposition by gluing the faces in a way that reverses the decomposition; namely, gluing white faces that arose from the same surface by the identity and folding shaded faces to form crossing circles. Half-twists at crossings are recovered at this step by gluing shaded faces between components instead of within a single component; see Figure \ref{fig:half_twist_glue}.
    \end{proof}

\begin{figure}[h]
            \begin{center}
            \includegraphics[]{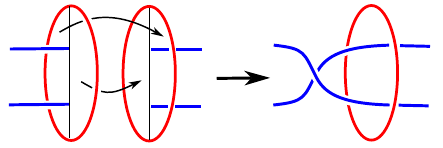}
            \end{center}
            \caption{Gluing the shaded faces of the bowtie decomposition to form a half-twist at a crossing circle, reversing steps 1 and 3.}
            \label{fig:half_twist_glue}
        \end{figure}

\begin{definition}\label{def:bowtie_decomp}
    The decomposition of the complement of a fully augmented link described in proposition \ref{prop:decomp} is the \emph{bowtie decomposition} of $M\setminus L$, after the distinctive shapes formed by the shaded faces. In full specificity, it is referred to as the bowtie decomposition of $\pi(L)$ on $\Sigma$ in $M$, but these are often omitted where the manifold and projection surface are clear.
\end{definition}

\begin{definition}
    Given a cellular fully augmented link $L$ with projection $\pi(L)$ to a surface $\Sigma$, apply the bowtie decomposition. The \emph{nerve} of $\pi(L)$, or of $L$ when the projection surface is clear, is the graph on $\Sigma$ which has a vertex for each white face and an edge between vertices wherever the associated white face(s) meet at an ideal vertex of the decomposition.
\end{definition}

\begin{lemma}\label{lem:no._white_faces}
    Let $L$ be a fully augmented link with $c$ crossing circles. The bowtie decomposition of a projection $\pi(L)$ that is cellular in a surface $\Sigma$ of genus $g$ has $c+2-2g$ white faces.
\end{lemma}

    \begin{proof}
        Proceed by argument of Euler characteristic.
        
        Each face of the nerve of $\pi(L)$ corresponds to a shaded face of the bowtie decomposition, of which there are $2c=F$, as each crossing circle gives rise to exactly two shaded faces. Each shaded face of the bowtie decomposition is triangular, and thus each face of the nerve is also a triangle with one edge through each vertex of the shaded face. Each edge is shared by two faces, so the number of edges is $\frac{3}{2}2c=3c=E$. The number of vertices $V$ is the number of white faces in the decomposition. The Euler characteristic $\chi(\Sigma)$ of a genus $g$ surface is $2-2g$.

        Thus, we have:
        \begin{align*}
            \chi(\Sigma)&=V-E+F\\
            2-2g&=V-3c+2c\\
            V&=c+2-2g.
        \end{align*}
    \end{proof}

\begin{lemma}\label{lem:tri_prism}
    The volume of an ideal hyperbolic triangular prism is bounded above by $3v_{tet}$, where $v_{tet}=1.01494...$ is the volume of the regular ideal hyperbolic tetrahedron.
\end{lemma}

    \begin{proof}
        Cut off one vertex of the triangular prism along the triangle through its three adjacent vertices; this forms one tetrahedron. What remains is a square-base pyramid, which may be decomposed into two further tetrahedra by splitting it along the triangle between the apex and a diagonal of the base. Each of these three ideal tetrahedra have volume bounded above by $v_{tet}$.
    \end{proof}

\begin{theorem}\label{thm:bounded_volume}
Let $\Sigma$ be a closed surface of genus at least two and $M$ be the trivial mapping torus $\Sigma\times S^1$. Let $L$ be a hyperbolic fully augmented link with projection $\pi(L)$ to $\Sigma\subset M$. The volume of the complement is bounded above by
\[\Vol(M\setminus L)\leq6(3c+2g-2)v_{tet}.\]
\end{theorem}

    \begin{proof}
        Apply the bowtie decomposition to $\pi(L)$. The result is a thickened surface $\Sigma\times I$ decorated identically on both boundaries with the decomposition of $\pi(L)$. Each white face is a (not necessarily geodesic) ideal $n$-gon that can be triangulated with $n-2$ ideal triangles. The bowtie decomposition has exactly $6c$ edges, as each crossing circle gives rise to exactly two triangular shaded faces, each of which is shared by a white face. By Lemma \ref{lem:no._white_faces}, there are $c+2-2g$ white faces; thus, triangulating all the white faces in this way produces $6c-2(c+2-2g)=4c+4g-4$ triangles. With the $2c$ shaded faces, each boundary is then decorated with an ideal triangulation with $6c+4g-4$ triangles. 

        Because both boundaries are decorated identically, every face of this triangulation matches exactly to a face on the opposite boundary of the thickened surface. Connect all matching vertices via an edge. Cutting along all the rectangles formed with the matching edges of the triangulation forms a triangular prism for every face of the triangulation. Each of these has volume bounded above by $3v_{tet}$ by Lemma \ref{lem:tri_prism}. As there are $6c+4g-4$ triangles on each boundary, the result follows.
    \end{proof}

\begin{remark}\label{rem:hyperelliptic_involution}
    By Theorem \ref{thm:infinite_volume}, there does not exist a linear bound on hyperbolic volume in terms of the number of crossing circles $c$ of a fully augmented link in any mapping torus $(\Sigma\times I)/\phi$, provided that there exists a nontrivial curve $\gamma\subset\Sigma$ such that $\phi(\gamma)\not\sim\pm\gamma$. By Theorem \ref{thm:bounded_volume}, links in $\Sigma\times S^1$, the trivial mapping torus, \emph{do} admit such a bound.
    
    There is an example of a map which is not covered by either of these two cases; that is, a map that is nontrivial, yet preserves the isotopy class of all simple closed curves --- the \emph{hyperelliptic involution} of a genus two surface. By \cite{hyperelliptic_involution}, this is the only such map. Neither of the techniques in this paper are conclusive on whether a link in the mapping torus of the genus two hyperelliptic involution admits a linear, or indeed any, upper volume bound.
\end{remark}

\bibliographystyle{amsplain}  
\bibliography{bibliography}

\end{document}